\documentclass[a4paper, oneside, reqno, 12pt]{amsart}
\usepackage[utf8]{inputenc}
\usepackage{amsmath, amssymb, amsfonts}
\usepackage{fullpage}
\usepackage{textcomp, cmap, comment}
\usepackage{mathtools} 
\usepackage[nobysame, initials]{amsrefs}
\usepackage{enumitem, tikz, caption}
\usepackage[T1]{fontenc}
\usepackage{subfig}

\usepackage{comment}
\usepackage{color, xcolor}

\theoremstyle{plain}
\newtheorem{theorem}{Theorem}
\newtheorem{conjecture}[theorem]{Conjecture}

\newtheorem{lemma}[theorem]{Lemma}

\theoremstyle{definition}

\theoremstyle{remark}
\newtheorem*{remark*}{Remark}

\DeclareMathOperator{\pr}{pr}

\DeclareMathOperator{\conv}{conv}

\usepackage{pgf,tikz, mathrsfs}
\usetikzlibrary{calc}
\usetikzlibrary{arrows}

\usepackage{hyperref}
\hypersetup{
	colorlinks   = true, 
	urlcolor     = blue, 
	linkcolor    = red, 
	citecolor   = green 
}

\title[A cap covering theorem]{A cap covering theorem}
\author[A.~Polyanskii]{{A.~Polyanskii}}
\address{Alexandr Polyanskii,
\newline\hphantom{iii} \href{https://mipt.ru/english/}{MIPT}
}
\email{\href{mailto:alexander.polyanskii@yandex.ru}{alexander.polyanskii@yandex.ru}}
\urladdr{\url{http://polyanskii.com}}
\thanks{The author was supported by the Russian Federation Government through Grant No. 075-15-2019-1926. The author is a Young Russian Mathematics award winner and would like to thank its sponsors and jury. 
}
\begin{document}
\thispagestyle{empty}
\begin{abstract}

A \textit{cap} of spherical radius $\alpha$ on a unit $d$-sphere $S$ is the set of points within spherical distance $\alpha$ from a given point on the sphere. Let $\mathcal F$ be a finite set of caps lying on $S$. We prove that if no hyperplane through the center of \( S \) divides $\mathcal F$ into two non-empty subsets without intersecting any cap in $\mathcal F$, then there is a cap of radius equal to the sum of radii of all caps in $\mathcal F$ covering all caps of $\mathcal F$ provided that the sum of radii is less $\pi/2$.

This is the spherical analog of the so-called Circle Covering Theorem by Goodman and Goodman and the strengthening of Fejes T\'oth's zone conjecture proved by Jiang and the author.

\end{abstract}
\maketitle
\section{Introduction}

A finite collection $\mathcal K$ of convex bodies in $\mathbb R^d$ is called \textit{non-separable} if any hyperplane intersecting $\conv \bigcup \mathcal K$ meets a convex body of $\mathcal K$. The following theorem was conjectured by Erd\H{o}s and proved by Goodman and Goodman~\cite{Goodman1945}. 

\begin{theorem}[A.\;W.~Goodman, R.\;E.~Goodman, 1945]
    \label{theorem:goodman}
    If a finite collection of disks of radii $r_1, \dots, r_n$ is non-separable, then it is possible to cover them by a disk of radius $r_1+\dots+r_n$.
\end{theorem}

Using the argument of Goodman and Goodman, it is easy to prove the high-dimensional analog of Theorem~\ref{theorem:goodman}. Moreover, Bezdek and Langi~\cite{Bezdek2016} noticed that the same argument works for a non-separable collection of homothets of a centrally symmetric convex body. Also, they investigated a similar question on covering of a non-separable collection of positive homothets of a convex body that is not necessary centrally symmetric. Later Akopyan, Balitskiy, and Grigorev~\cite{Akopyan2017} improved their result. The aim of the current note is to prove the spherical analog of Theorem~\ref{theorem:goodman}; see~\cite[Conjecture~3]{Jiang2017}.

In order to state our result, we need several definitions. Denote by $S$ the unit $d$-sphere embedded in $\mathbb R^{d+1}$ centered at the origin. A \textit{cap} of spherical radius $\alpha$ on $S$ is defined as the set of points within spherical distance $\alpha$ from a given point on the sphere. A \textit{great sphere} is the intersection of $S$ and a hyperplane passing through the origin. We call a great sphere \textit{avoiding} for a collection of caps if it does not intersect any cap in this collection. A finite collection of caps is called \textit{non-separable} if no avoiding great sphere divides the collection into two non-empty sets.

\begin{theorem}
\label{theorem:mainresult}
Let $\mathcal F$ be a non-separable collection of caps of spherical radii $\alpha_1, \dots, \alpha_n$. If $\alpha_1+\dots+\alpha_n < \pi /2$, then $\mathcal F$ can be covered by one cap of radius $\alpha_1+\dots+\alpha_n$.
\end{theorem}

Recall that a pair of antipodal caps can be viewed as the dual of a zone, where a \textit{zone} of width $2\alpha$ on $S$ is the set of points within spherical distance $\alpha$ from a given great sphere. (The projective duality in $\mathbb R^{d+1}$ interchanges a line through the origin with its orthogonal hyperplane through the origin, that is, a pair of antipodal points is the dual of a great sphere.) It is worth mentioning that Theorem~\ref{theorem:mainresult} for $\alpha_1+\dots +\alpha_n=\pi/2$ under some technical assumptions is a corollary of so-called Fejes T\'oth's zone conjecture~\cite{Tth1973}
that is proved in~\cite[Theorem~1 and Corollary~3]{Jiang2017}; see also the recent work~\cite{ortegamoreno2019optimal} of Ortega-Moreno, where the conjecture is confirmed for zones of the same width.

\begin{theorem}[Jiang, Polyanskii, 2017]
    \label{theorem:tothconjecture}
    Given a collection of zones covering \( S \), the sum of width of all zones in the collection is at least \( \pi \).
\end{theorem}

The proof of Theorem~\ref{theorem:mainresult} heavily relies on ideas developed in the context of studying planks~\cites{tarski1932uwagi, Bang1951, Ball1991} covering a convex body, where a \textit{plank} (or \textit{slab}, or \textit{strip}) of width $w$ is a set of all points lying between two parallel hyperplanes in $\mathbb R^d$ at distance $w$. The connection between planks and zones is obvious: A zone of width $2\alpha$ is the intersection of $S$ and the centrally symmetric about the origin plank of width $2\sin \alpha$. Another key idea of our proof is considering the farthest point Voronoi diagram of the so-called Bang set (see~\eqref{equation:Bang_set}). This idea appeared in the very recent work~\cite{balitskiy2020multiplank} of Balitskiy, where trying to understand the proof of the theorem of Kadets~\cite{Kadets2005} on the sum of inradii of convex bodies in a collection covering a unit ball, he introduced a new concept of \textit{multiplanks}. Since we do not use this involved concept in its greatest generality, we decided to give only a short remark in the discussion section about relation of our proof to this concept. Nevertheless, we highly recommend an interested reader to understand it: We believe that it will be very helpful in proving new results on coverings.

\subsection*{Acknowledgements} The author thanks Alexey Balitskiy for the fruitful discussions of his work~\cite{balitskiy2020multiplank}. Besides, the author thanks the anonymous referees whose remarks helped fix some errors and improve the presentation.

\section{Covering zones instead of caps}
\label{section:proof}

From now on, we consider only open centrally symmetric about the origin planks, and thus we omit the phase `open centrally symmetric about the origin' most of the time. For a plank \( P \), denote by \( \mathbf w (P) \) a vector \( \mathbf w \) such that \( P=\{\mathbf x\in \mathbb R^d: | \langle \mathbf x, \mathbf w \rangle | < \langle \mathbf w, \mathbf w \rangle  \}\). For a zone $Z$, set $\mathbf w(Z)=\mathbf w(P)$, where $P$ is the open plank such that $Z$ coincides with the closure of $S\cap P$.

The main tool of the current paper is the following dual reformulation of \cite[Lemma~4]{Jiang2017}.

\begin{lemma}
\label{mainlemma:equivalent}
Let $Z_1, \dots, Z_n\subset S$ be zones of width $2\alpha_1,\dots, 2\alpha_n$, respectively, such that $\alpha=\alpha_1+\dots+\alpha_n\leq \pi/2$. Set $\mathbf w_i:=\mathbf w(Z_i)$. If  $\mathbf w=\sum_{i=1}^n \mathbf w_i$ satisfies the following inequalities
\[
|\mathbf w| \geq  \sin \alpha \text{ and }|\mathbf w-\mathbf w_i| \leq  \sin (\alpha -\alpha_i)
\text{ for all } i\in [n],
\]
then there is a zone of width $2 \alpha$  covering $Z_1,\dots, Z_n$.
\end{lemma}

We first show that the following theorem (Theorem~\ref{theorem:equivalent}) implies Theorem~\ref{theorem:mainresult}. We remark that Theorem~\ref{theorem:mainresult} also easily implies Theorem~\ref{theorem:equivalent}.
\begin{theorem}
\label{theorem:equivalent}
Let $Z_1, \dots, Z_n\subset S$ be zones of width $2\alpha_1, \dots, 2\alpha_n$, respectively, such that $\alpha_1+\dots+\alpha_n< \pi/2$. If $S\setminus \bigcup_{i=1}^n Z_i$ has at most one pair of two antipodal connected components, then the zones $Z_1, \dots, Z_n$ can be covered by a zone of width $2\alpha_1+\dots+2\alpha_n$.
\end{theorem}
\begin{proof}[Theorem~\ref{theorem:equivalent} implies Theorem~\ref{theorem:mainresult}]
\phantom{\qedhere}
Suppose that spherical caps $D_1, \dots, D_n$ satisfy the conditions of Theorem~\ref{theorem:mainresult}. Let $D'_i$ be an open cap concentric with $D_i$ of spherical radius $\pi/2-\alpha_i$. By projective duality, the center of an open hemisphere  lies in $D_i'$ if and only if this hemisphere covers $D_i$. Since no avoiding great sphere divides $\{D_1, \dots, D_n\}$ into two non-empty sets, by projective duality, we get
\[
    \bigcap_{i=1}^n\varepsilon_i D_i'=\emptyset\text{ unless all } \varepsilon_i\in \{\pm1\} \text{ are the same}.
\]
Hence the zones $S\setminus \left(D_1' \cup (-D_1')\right)$,\dots, $S\setminus \left(D_n'\cup (-D_n')\right)$ satisfy the conditions of Theorem~\ref{theorem:equivalent}. Therefore, they can be covered by a zone $Z$ of width $2\alpha_1+\dots+2\alpha_n$. Since $S\setminus Z$ is the union of two antipodal open caps $D'$ and $-D'$ of radii $\pi/2-(\alpha_1+\dots+\alpha_n)$, without loss of generality we obtain
\[
D' \subseteq \left(\bigcap_{i=1}^n D_i'\right).
\]
By projective duality, the closed cap concentric with $D'$ of radius $\alpha_1+\dots+\alpha_n$ covers caps $D_1,\dots, D_n$. 
\end{proof}

Having shown that Theorem~\ref{theorem:equivalent} implies Theorem~\ref{theorem:mainresult}, we now proceed with the proof of Theorem~\ref{theorem:equivalent}.

\begin{proof}[Proof of Theorem~\ref{theorem:equivalent}]
Denote by $P_i$ the open plank such that the intersection of its closure with $S$ is $Z_i$. Set $\mathbf w_i=\mathbf w(P_i)$ for all $i\in[n]$.
Without loss of generality let us assume that $\mathbf w=\mathbf w_1+\dots +\mathbf w_n$ has the maximum norm among vectors of the \textit{Bang set}
\begin{equation}
    \label{equation:Bang_set}
    L=\left\{\sum_{i=1}^n\pm \mathbf w_i\right\}.
\end{equation}

First, let us show that we can assume $|\mathbf w|\leq \sin (\alpha_1+\dots+\alpha_n)$. Indeed, suppose that $|\mathbf w|>\sin (\alpha_1+\dots+\alpha_n)$. Thus the family of all subsets $J\subset [n]$ such that 
\[
\left |\sum_{i\in J}\mathbf w_i\right|>\sin \left( \sum_{i\in J} \alpha_i\right)
\]
is non-empty. Choose among them a minimal subset $I$. Since $|\mathbf w_i|=\sin \alpha_i$, we have $|I|>1$, and so we can apply Lemma~\ref{mainlemma:equivalent} to $I$ and cover zones $Z_i$, $i\in I$, by the zone $Z$. Replacing the zones $Z_i$, $i\in I$, by $Z$, we obtain a new collection of zones with the same sum of width covering the original collection of zones. Put $S\setminus Z_i=D_i\cup (-D_i)$ and $S \setminus Z = D \cup (-D)$, where $D_i$ and $D$ are open caps such that $D\subset \bigcap_{i\in I} D_i$. By the conditions of the theorem, we can assume that 
\[
\bigcap_{i=1}^n \varepsilon_i D_i=\emptyset \text{ unless all } \varepsilon_i \text{ are the same.}
\]
Since 
\[
\bigcap_{i\in [n]\setminus I} \varepsilon_i D_i \cap \varepsilon D \subseteq \left(\bigcap_{i\in [n]\setminus I} \varepsilon_i D_i\right) \cap \left(\bigcap_{i\in I} \varepsilon D_i\right),
\]
we get that the new collection of zones satisfies the conditions of the theorem. 
Therefore, we can reduce the number of zones and assume that $|\mathbf w|\leq \sin (\alpha_1+\dots+\alpha_n)$.

Next, we use the following version of so-called Bang's Lemma. It seems that it was first published in a similar form by Fenchel~\cite{fenchel1951th}.
\begin{lemma}[Bang's Lemma]
\label{lemma:farthestpoint}
If \( \mathbf t \) has maximum norm among elements of \( \mathbf t + \mathbf x - L \) for \( \mathbf x=\sum_{i=1}^n \varepsilon_i \mathbf w_i \in L \), then
\[ 
\mathbf t \in M_{\mathbf x}:=\bigcap_{i=1}^n\left\{\mathbf y\in \mathbb R^d: \langle \mathbf y, \mathbf -\varepsilon_i \mathbf w_i \rangle \geq \langle \mathbf w_i, \mathbf w_i \rangle \right\}\subseteq \mathbb R^d\setminus \bigcup_{i=1}^n P_i.
\]
\end{lemma}
\begin{proof}
Suppose that $\mathbf t \in \{\mathbf y\in \mathbb R^d: \langle \mathbf y,  -\varepsilon_i \mathbf w_i \rangle < \langle \mathbf w_i, \mathbf w_i \rangle \}$ for some $i\in [n]$. Then the vector \( \mathbf t + 2\varepsilon_i \mathbf w_i\in \mathbf t + \mathbf x - L \) is longer than $\mathbf t$ (see Figure~\ref{hyperplane}), a contradiction.
\end{proof}

Let \( B \) be the unit open ball with center at the origin. For \( \mathbf x \in L \), consider the set
\[
    A_{\mathbf x}= \big\{ \mathbf t\in B: \mathbf t \text{ has maximum norm among of elements of } \mathbf t + \mathbf x - L \big\}. 
\]
Denote by \( \pr:\mathbb R^d\setminus \{\mathbf o\}\to S \) the central projection onto \( S \), where \( \mathbf o \) is the origin. By Lemma~\ref{lemma:farthestpoint}, we have 
\(
\pr (A_{\mathbf x})
\subseteq
\pr (M_{\mathbf x}\cap B )\subseteq \pr \big(B\setminus \bigcup_{i=1}^n P_i\big)
\).
Since \( \pr (M_{\mathbf x} \cap B)\) is an open set, we obtain \( \pr (A_{\mathbf x}) \) is a subset \( S\setminus \bigcup_{i=1}^n Z_i \). (Recall that \( Z_i \) is the closure of \( S\cap P_i \).)

We claim that
\begin{equation}
\label{equation:goal}
A_{\mathbf w}=\{ \mathbf t\in B: \langle \mathbf t, -\mathbf w \rangle \geq \langle \mathbf w, \mathbf w \rangle \} = \{\mathbf t \in B:\|\mathbf t\| \geq \| \mathbf t + 2\mathbf w\|\}.
\end{equation}
Before proving (\ref{equation:goal}), let us show how to finish the proof of the theorem. Since $|\mathbf w|\leq \sin (\alpha_1+\dots+\alpha_n)<1$, the set \( \pr (A_{\mathbf w})\) is an open cap $X$ of radius at least $\pi/2-(\alpha_1+\dots+\alpha_n)$ lying in $S\setminus \bigcup_{i=1}^n Z_i$; see Figure~\ref{figure2:proof_thm6}. Therefore, the zone $Z:=S\setminus (X\cup (-X))$ of width at most $2\alpha_1+\dots+2\alpha_n$ covers $Z_1,\dots, Z_n$. So, to finish the proof, it is enough to show (\ref{equation:goal}).

\begin{figure}[h]
  \begin{minipage}{.49\textwidth}
    \begin{tikzpicture}[scale=1.6]
      \coordinate (O) at (3.5,0);
      \coordinate (A) at (0.5,1.5);
      \coordinate (B) at (0.5,-1);
      \draw[darkgray, dashed] (0,-1.25) -- (4,-1.25);
      \draw[darkgray, dashed] (0,1.25) -- (4,1.25);
      \draw[darkgray, dashed] (0,0) -- (4,0);
      \draw[darkgray] (3.5, 0.2) -- (3.7, 0.2) -- (3.7, 0);
      \fill[darkgray] (O) circle (0.05);
      \fill[darkgray] (A) circle (0.05);
      \fill[darkgray] (B) circle (0.05);
      \draw[-latex, darkgray, shorten >= 0.05] (O) -- (A);
      \draw[-latex, darkgray, shorten >= 0.05] (B) -- (A);
      \draw[-latex, darkgray, shorten >= 0.05] (O) -- (3.5,1.25);
      \draw[-latex, darkgray, shorten >= 0.05] (O) -- (3.5,-1.25);
      \draw[-latex, darkgray, shorten >= 0.05] (O) -- (B);
      \node at (0.8, 1.2) [above right] {\( \mathbf t + 2 \varepsilon_i \mathbf w_i \)};
      \node at (0.6, -1.05) [right] {$\mathbf t$};
      \node at (-0.05,0.05) [above] {$2\varepsilon_i \mathbf w_i$};
      \node at (O) [below right] {\(\mathbf o\)};
      \node at (3.5,0.9) [left] {\( \varepsilon_i \mathbf w_i\)};
      \node at (3.5,-0.9) [left] {\( - \varepsilon_i \mathbf w_i\)};
      \node at (0,1.6) {};
      \node at (0,-1.5) {};
    \end{tikzpicture}    
\captionof{figure}{Proof of Lemma~\ref{lemma:farthestpoint}}
    \label{hyperplane}
  \end{minipage}
  \begin{minipage}{.45\textwidth}
  \centering
  \definecolor{mygray}{gray}{0.7}
\definecolor{mygrays}{gray}{0.8}
\begin{tikzpicture}[scale=0.5]
\draw [line width=0.5pt, gray] (0,0) circle (5);
\draw [line width=1pt, fill=mygrays] (3,4) arc (90-asin(3/5):90+asin(3/5):5cm) -- cycle;
\node at (0,4.5) { \( A_{-\mathbf w} \)};

\draw [line width=1pt, fill=mygray] (-3,-4) arc (270-asin(3/5):270+asin(3/5):5cm) -- cycle;
\node at (0,-4.5) {\( A_{\mathbf w} \)};

\draw [fill=black] (0,0) circle (1pt); 
\draw [fill=black] (0,4) circle (1pt) node[below right] {$\mathbf w$}; 
\draw [fill=black] (0,-4) circle (1pt) node[above right] {$-\mathbf w$};

\draw [->, thick, gray] (0,0) -- (0,4);
\draw [->, thick, gray] (0,0) -- (0,-4);
\end{tikzpicture}
\captionof{figure}{Proof of Theorem~\ref{theorem:equivalent}}
\label{figure2:proof_thm6}
  \end{minipage}

\end{figure}

First, we prove that
\begin{equation}
\label{equation:empty_condition}
A_{\mathbf x}=\emptyset \text{ for all } \mathbf x \in L\setminus \{ \pm \mathbf w\} .
\end{equation}
Indeed, using the fact that any two sets $M_{\mathbf x}$ for $\mathbf x\in L$ are strictly separated by one of the planks \( P_i \), any two sets \( \pr (A_{\mathbf x} ) \) for \( \mathbf x\in L \) are strictly separated subsets of \( S \setminus \bigcup_{i=1}^n Z_i \) consisting of at most two connected antipodal regions. Therefore, among sets \( A_{\mathbf x} \) for \( \mathbf x \) there are at most two non-empty sets. Since \( \mathbf w \in A_{- \mathbf w} \) and \( -\mathbf w\in A_{\mathbf w}\),  we conclude (\ref{equation:empty_condition}).

Next, we show that 
\begin{equation}
    \label{equation:include}
    \text{ if } \mathbf y - \mathbf w , \mathbf y + \mathbf w \in  B ,\text{ then } \mathbf y - \mathbf x \in B \text{ for all } \mathbf x \in L.
\end{equation} 
Suppose the contrary: For some \( \mathbf x \in L\setminus \{ \pm \mathbf w\} \), the point \( \mathbf y - \mathbf x \) does not lie in \( B \). Choosing a proper \( 0<\lambda<1\) and using convexity of \( B \), we get that \( \lambda \mathbf y -\mathbf x \in B \) for all \( \mathbf x\in L \) and \( \lambda \mathbf y - \mathbf x'\in A_{\mathbf x'}\) for some \( \mathbf x'\in L\setminus \{\pm \mathbf w\} \), a contradiction with (\ref{equation:empty_condition}).

Using~(\ref{equation:empty_condition}) and (\ref{equation:include}), we conclude that if the points \( \mathbf y - \mathbf w \) and \( \mathbf y + \mathbf w \) lie in \( B \), then \( \mathbf y - \mathbf w \in A_{\mathbf w } \) or \( \mathbf y + \mathbf w \in A_{-\mathbf w}\). Therefore, we obtain (\ref{equation:goal}).
\end{proof}

\section{Discussion}
\label{section:discussion}
First, we discuss the connection of our proof with the concept of multiplank.
\begin{remark*}
Using the terminology~\cite[Definition~2.1]{balitskiy2020multiplank} of Balitskiy, it is easy to see that the set 
\[ 
P_{L}:= \mathbb R^d\setminus \bigcup_{\mathbf x\in L} \big\{ \mathbf t\in \mathbb R^d: \mathbf t \text{ has maximum norm among of elements of } \mathbf t + \mathbf x - L \big\}
\]
is the open \textit{multiplank} of the set $L$ covering $\bigcup_{i=1}^n P_i$; see \cite[Proposition~4.1]{balitskiy2020multiplank}. In some sense our proof of Theorem~\ref{theorem:mainresult} relies on the fact
$P_L\cap B =P\cap B$, where $P$ is the plank with $\mathbf w(P)=\mathbf w$; see the stratification of a multiplank in the general case in \cite[Theorem~4.5]{balitskiy2020multiplank}. 
\end{remark*}

We recall the following problem resembling Theorem~\ref{theorem:mainresult}. It was proved for $\alpha\geq \pi/2$ in~\cite[Theorem~6.1]{bezdek2010covering} but still open for $\alpha< \pi/2$; see also~\cite[Problem~6.2]{Akopyan2012}.

\begin{conjecture}
If a cap of spherical radius $\alpha$ is covered with a
collection of convex spherical domains, 
then the sum inradii of all domains in the collection is at least $\alpha$.
\end{conjecture}

We finish with the following conjecture generalizing Theorems~\ref{theorem:tothconjecture} and~\ref{theorem:equivalent} posed by Maxim Didid.

\begin{conjecture}
\label{conjecture:didin}
Let \( Z_1, \dots, Z_n \subset S \) be zones of width \( \beta_1, \dots, \beta_n \), respectively. If \( S \setminus \bigcup_{i=1}^n Z_i \) consists of convex connected components \( L_1, \dots, L_{2m}\) with inradius \(\gamma_1, \dots, \gamma_{2m}\), respectively, then \(\beta_1+\dots + \beta_n+\gamma_1+\dots+\gamma_{2m}\geq 2\pi \).
\end{conjecture}

\bibliographystyle{alpha}
\bibliography{references}

\begin{bibdiv}
\begin{biblist}

\bib{Akopyan2017}{article}{
      author={Akopyan, Arseniy},
      author={Balitskiy, Alexey},
      author={Grigorev, Mikhail},
       title={{On the Circle Covering Theorem by A.W.~Goodman and
  R.E.~Goodman}},
        date={2017-03},
     journal={Discrete {\&} Computational Geometry},
      volume={59},
      number={4},
       pages={1001\ndash 1009},
         url={https://doi.org/10.1007/s00454-017-9883-x},
        note={Available at \url{https://arxiv.org/abs/1605.04300}},
}

\bib{Akopyan2012}{article}{
      author={Akopyan, Arseniy},
      author={Karasev, Roman},
       title={{Kadets-Type Theorems for Partitions of a Convex Body}},
        date={2012-07},
     journal={Discrete {\&} Computational Geometry},
      volume={48},
      number={3},
       pages={766\ndash 776},
         url={https://doi.org/10.1007/s00454-012-9437-1},
        note={Available at \url{https://arxiv.org/abs/1106.5635}},
}

\bib{balitskiy2020multiplank}{misc}{
      author={Balitskiy, Alexey},
       title={A multi-plank generalization of the {B}ang and {K}adets
  inequalities},
        date={2020},
        note={To appear in Israel Journal of Mathematics. Available at
  \url{https://arxiv.org/abs/2003.06707}},
}

\bib{Ball1991}{article}{
      author={Ball, Keith},
       title={The plank problem for symmetric bodies},
        date={1991-12},
     journal={Inventiones mathematicae},
      volume={104},
      number={1},
       pages={535\ndash 543},
         url={https://doi.org/10.1007/bf01245089},
        note={Available at \url{https://arxiv.org/abs/math/9201218}},
}

\bib{Bang1951}{article}{
      author={Bang, Thoger},
       title={{A Solution of the "Plank Problem"}},
        date={1951-12},
     journal={Proceedings of the American Mathematical Society},
      volume={2},
      number={6},
       pages={990},
         url={https://doi.org/10.2307/2031721},
}

\bib{Bezdek2016}{article}{
      author={Bezdek, K{\'{a}}roly},
      author={L{\'{a}}ngi, Zsolt},
       title={{On Non-separable Families of Positive Homothetic Convex
  Bodies}},
        date={2016-08},
     journal={Discrete {\&} Computational Geometry},
      volume={56},
      number={3},
       pages={802\ndash 813},
         url={https://doi.org/10.1007/s00454-016-9815-1},
        note={Available at \url{https://arxiv.org/abs/1602.01020}},
}

\bib{bezdek2010covering}{article}{
      author={Bezdek, K{\'a}roly},
      author={Schneider, Rolf},
       title={Covering large balls with convex sets in spherical space},
        date={2010},
     journal={Beitr{\"a}ge Algebra Geometry},
      volume={51},
      number={1},
       pages={229\ndash 235},
        note={Available at \url{https://arxiv.org/abs/1110.4316}},
}

\bib{fenchel1951th}{article}{
      author={Fenchel, W},
       title={On th. bang's solution of the plank problem},
        date={1951},
     journal={Matematisk tidsskrift. B},
       pages={49\ndash 51},
}

\bib{Goodman1945}{article}{
      author={Goodman, A.~W.},
      author={Goodman, R.~E.},
       title={{A Circle Covering Theorem}},
        date={1945-11},
     journal={The American Mathematical Monthly},
      volume={52},
      number={9},
       pages={494\ndash 498},
         url={https://doi.org/10.1080/00029890.1945.11999187},
}

\bib{Jiang2017}{article}{
      author={Jiang, Zilin},
      author={Polyanskii, Alexandr},
       title={Proof of {L\'{a}}szl{\'{o}} {F}ejes {T\'{o}}th's zone
  conjecture},
        date={2017-11},
     journal={Geometric and Functional Analysis},
      volume={27},
      number={6},
       pages={1367\ndash 1377},
         url={https://doi.org/10.1007/s00039-017-0427-6},
        note={Available at \url{https://arxiv.org/abs/1703.10550}},
}

\bib{Kadets2005}{article}{
      author={Kadets, Vladimir},
       title={Coverings by convex bodies and inscribed balls},
        date={2005-05},
     journal={Proceedings of the American Mathematical Society},
      volume={133},
      number={5},
       pages={1491\ndash 1495},
         url={https://doi.org/10.1090/s0002-9939-04-07650-6},
        note={Available at \url{https://arxiv.org/abs/math/0312133}},
}

\bib{ortegamoreno2019optimal}{article}{
      author={Ortega-Moreno, Oscar},
       title={An optimal plank theorem},
        date={2021},
     journal={Proceedings of the American Mathematical Society},
      volume={149},
      number={3},
       pages={1225\ndash 1237},
        note={Available at \url{https://arxiv.org/abs/1906.04126}},
}

\bib{tarski1932uwagi}{article}{
      author={Tarski, Alfred},
       title={Uwagi o stopniu r{\'o}wnowa{\.z}no{\'s}ci wielok\k{a}t{\'o}w
  [{R}emarks on the degree of equivalence of polygons]},
        date={1932},
     journal={Parametr},
      number={2},
       pages={310\ndash 314},
}

\bib{Tth1973}{article}{
      author={T{\'{o}}th, L.~Fejes},
       title={{Exploring a Planet}},
        date={1973-11},
     journal={The American Mathematical Monthly},
      volume={80},
      number={9},
       pages={1043\ndash 1044},
         url={https://doi.org/10.1080/00029890.1973.11993441},
}

\end{biblist}
\end{bibdiv}

\end{document}